\documentclass{article}
\usepackage[utf8]{inputenc}
\pdfoutput=1 
\usepackage{caption}
\usepackage[english]{babel}
\usepackage{enumitem}
\usepackage{amsmath,amsfonts,amssymb,amsthm}
\usepackage[mathscr]{euscript}
\usepackage[table]{xcolor}
\usepackage{graphicx}
\usepackage[numbers]{natbib}

\def\figurescale{0.6}

\newcommand{\Sx}[1]{\mathscr{S}_{#1}}
\newcommand{\Sn}[0]{\Sx{n}}

\newcommand{\M}[0]{\mathscr{M}}
\newcommand{\Mx}[3]{\mathscr{M}_{#1}^{#2}(#3)}

\newcommand{\Desc}[0]{\mathscr{D}}

\newtheorem{theorem}{Theorem}[section]
\newtheorem{corollary}[theorem]{Corollary}
\newtheorem{lemma}[theorem]{Lemma}
\newtheorem{proposition}[theorem]{Proposition}

\theoremstyle{definition}
\newtheorem{definition}[theorem]{Definition}

\theoremstyle{remark}

\newtheoremstyle{case}
	{\topsep}                    % Space above
    {\topsep}                    % Space below
    {\normalfont}                % Body font
    {}                           % Indent amount
    {\itshape}                   % Theorem head font
    {.}                          % Punctuation after theorem head
    {.5em}                       % Space after theorem head
    {}  % Theorem head spec (can be left empty, meaning ‘normal’)
    
\theoremstyle{case}
\newtheorem{case}{Case}

\title{Major index distribution over permutation classes}
\author{Michal Opler\thanks{Computer Science Institute, Charles University in Prague, \texttt{\href{mailto:opler@iuuk.mff.cuni.cz}{opler@iuuk.mff.cuni.cz}}.}}
\date{\today}

\begin{document}
\maketitle
\begin{abstract}
For a permutation~$\pi$ the major index of~$\pi$ is the sum of all indices~$i$ such that $\pi_i > \pi_{i+1}$.
It is well known that the major index is equidistributed with the number of inversions over all permutations of length~$n$.
In this paper, we study the distribution of the major index over pattern-avoiding permutations of length~$n$.
We focus on the number $M_n^m(\Pi)$ of permutations of length~$n$ with major index~$m$ and avoiding the set of patterns~$\Pi$.

First we are able to show that for a singleton set $\Pi = \{\sigma\}$ other than some trivial cases, the values $M_n^m(\Pi)$ are monotonic in the sense that $M_n^m(\Pi) \leq M_{n+1}^m(\Pi)$.
Our main result is a study of the asymptotic behaviour of~$M_n^m(\Pi)$ as~$n$ goes to infinity.
We prove that for every fixed m and~$\Pi$ and~$n$ large enough, $M_n^m(\Pi)$ is equal to a polynomial in~$n$ and moreover, we are able to determine the degrees of these polynomials for many sets of patterns. 
\end{abstract}

\section{Introduction}
\label{sec:introduction}

Let $\Sn$ be the set of permutations of the letters $\{1,2, \ldots, n\} = [n]$.
We write a permutation~$\pi \in \Sn$ as a sequence $\pi_1 \cdots \pi_n$.
A \textsl{permutation statistic} is a function $st: \Sn \to \mathbb{N}_0$.
For a permutation $\pi$, an \textsl{inversion} is a pair of different indices $i < j$ such that $\pi_i > \pi_j$ and the \textsl{number of inversions} is denoted by $inv(\pi)$.
The number of inversions is the oldest and best-known permutation statistic.
Already in 1838, Stern~\cite{Stern1838} proposed a problem of how many inversions there are in all the permutations of length~$n$.
The distribution of the number of inversions was given shortly after that by Rodrigues~\cite{Rodrigues1839}.

However, we will focus on a different well-known permutation statistic in this paper.
For a permutation~$\pi$, we say that there is a \textsl{descent} on the $i$-th position if $\pi_i > \pi_{i+1}$.
The \textsl{major index} of $\pi$, denoted by $maj(\pi)$, is then the sum of the positions, where the descents occur.
The major index statistic is younger than the number of inversions, as it was first defined by MacMahon~\cite{MacMahon1916} in 1915.
Among other results, MacMahon proved its equidistribution with the number of inversions by showing that their generating functions are equal and started the systematic study of permutation statistics in general.
That is why we call the statistics equidistributed with the number of inversions Mahonian.
Then it took a long time before Foata~\cite{Foata1968} proved the equidistribution by constructing his famous bijection.
Since then many new Mahonian statistics appeared in the literature, most of which are included in the classification given by Babson and Steingrímsson~\cite{Babson2000}.
For the actual values of Mahonian statistics' distribution see the Mahonian numbers sequence A008302~\cite{oeisA008302}.

We say that two sequences $a_1 \cdots a_n$ and $b_1 \cdots b_n$ are \textsl{order-isomorphic} if the permutations required to sort them are the same.
A permutation~$\pi$ \textsl{contains} a pattern~$\sigma$ if there is a subsequence of $\pi_1 \cdots \pi_n$ order-isomorphic to~$\sigma$.
Otherwise we say that~$\pi$ \textsl{avoids} the pattern~$\sigma$.
Pattern avoidance is an active area of research in combinatorics and although the systematic study of pattern avoidance started relatively recently, there is already an extensive amount of literature.
A good illustration of an application of pattern avoidance in computer science is that stack-sortable permutations are exactly the ones avoiding pattern 231, which was proved by Knuth~\cite{Knuth1969}.

Let $\Sn(\sigma)$ be the set of permutations of length~$n$ avoiding~$\sigma$ and $S_n(\sigma)$ its cardinality.
We say that patterns~$\sigma$ and~$\tau$ are \textsl{Wilf-equivalent} if $S_n(\sigma) = S_n(\tau)$ for every~$n$.
For a permutation statistic st, we say that patterns $\sigma$ and $\tau$ are \textsl{st-Wilf-equivalent} if there is a bijection between $\Sn(\sigma)$ and $\Sn(\tau)$ which preserves the statistic st.
This refinement of Wilf equivalence has been extensively studied for short patterns of length 3, see \cite{Bloom2009, Bloom2010, Elizalde2011, Robertson2002}.
An exhaustive classification of Wilf-equivalence and permutation statistics among these patterns was given by Claesson and Kitaev~\cite{Claesson2008}.
On the other hand, not much is known about permutation statistics and patterns of length 4 and greater.
Recently, Dokos et al.~\cite{Dokos2012} presented an in-depth study of major index and number of inversions including st-Wilf-equivalence.
They conjectured maj-Wilf-equivalence between certain patterns of length 4, which was proved by Bloom~\cite{Bloom2014}.
Another conjecture from Dokos et al. concerning maj-Wilf-equivalent patterns of a specific form was partly proved by Ge, Yan and Zhang~\cite{Yan2015}.

Claesson, Jelínek and Steingrímsson~\cite{Jelinek2012} analysed the inversion number distribution over pattern-avoiding classes.
Let $I^k_n(\sigma)$ be the number of $\sigma$-avoiding permutations with length~$n$ and~$k$ inversions.
Claesson et al. studied $I^k_n(\sigma)$ for a fixed~$k$ and a single pattern~$\sigma$ as a function of~$n$.
Our goal is to provide similar analysis for the distribution of major index.

For a pattern~$\sigma,$ let $\M^m_n(\sigma)$ be the set of $\sigma$-avoiding permutations with length~$n$ and major index~$m$, and let $M^m_n(\sigma)$ denote its cardinality.
For a set of patterns~$\Pi$, let $\M^m_n(\Pi) = \bigcap_{\sigma \in \Pi} \M^m_n(\sigma)$ and $M^m_n(\Pi)$ its cardinality.
Claesson et al.~\cite{Jelinek2012} conjectured that $I^k_n(\sigma) \leq I^k_{n+1}(\sigma)$ for every $k$, $n$ unless $\sigma$ is an increasing pattern (i.e. a pattern of the form $1 \cdots l$).
In Section~\ref{sec:fixed-major}, we will prove the analogous claim for major index by constructing an injective mapping $f \colon M^m_n(\sigma) \to M^m_{n+1}(\sigma)$ for every $\sigma \neq 12  \cdots l$.
Furthermore, we show that the claim does not hold in general for an arbitrary set of multiple patterns.

In Section~\ref{sec:asymptotics}, we focus on the asymptotic behaviour of~$M^m_n(\Pi)$ for a fixed~$m$ and~$\Pi$ as $n$ goes to infinity.
We note that the asymptotic behaviour for the number of inversions is known only for sets avoiding a single pattern.
In contrast, our results apply to general (possibly infinite) set of patterns.
It turns out that the values $M^m_n(\Pi)$ are eventually equal to a polynomial in~$n$, which is consistent with the behaviour of $I^k_n(\sigma)$.
The natural question to ask is how the degrees of these polynomials depend on~$\Pi$ and~$m$.

Let $deg(m, \Pi)$ be the degree of the polynomial $P$ such that $P(n) = M^m_n(\Pi)$ for $n \geq n_0$.
Similarly, let $deg_{\text{I}}(k, \sigma)$ be the degree of the polynomial~$P$ such that $P(n) = I^k_n(\sigma)$ for $n \geq n_0$.
In the case of the number of inversions, there are just two types of patterns.
For a pattern $\sigma$, we have either $deg_{\text{I}}(k,\sigma)=k$ for every~$k$, or there is a constant~$c$ such that $deg_{\text{I}}(k,\sigma)=\min(k,c)$.
All these results about $I^k_n(\sigma)$ and $deg_{\text{I}}(m, \sigma)$ were shown in the aforementioned paper by Claesson et al.~\cite{Jelinek2012}.

However, the situation gets more complicated when dealing with major index.
We show how $deg(m, \lbrace \sigma \rbrace)$ depends on the structure of $\sigma$ and determine $deg(m, \Pi)$ for many types of~$\Pi$, including all the cases when~$\Pi$ is a singleton set.
There are still patterns~$\sigma$ for which $deg(m, \lbrace \sigma \rbrace) = m$, but for many patterns $deg(m, \lbrace \sigma \rbrace)$ is a complicated function of~$m$ which tends to infinity slower than linearly (approximately as $\sqrt{m}$).
Note that there are unfortunately sets~$\Pi$ for which we are not able to precisely determine~$deg(m, \Pi)$.
In these cases, our results provide at least an upper bound.

Finally, we conclude Section~\ref{sec:asymptotics} by using our results to show that the asymptotic probability of a random permutation with major index~$m$ avoiding~$\Pi$ is either~0 or~1.
This again corresponds with the number of inversions, where the analogous claim was proved for singleton sets of patterns.

\section{Preliminaries}
\label{sec:basics}
In this section, we recall some standard notions related to permutation patterns and introduce a simple decomposition of permutations.

%\section{Standard notions}

Let $\Sn$ be the set of permutations of the letters $\{1,2, \ldots, n\} = [n]$.
A~permutation $\sigma \in \Sn$ will be represented as a sequence of its values $\sigma = \sigma_1 \sigma_2 \cdots \sigma_n$, where $\sigma_i = \sigma(i)$.
We say that two sequences of integers $a_1 \cdots a_k$ and $b_1 \cdots b_k$ are \textsl{order-isomorphic} if for every $i, j \in [k]$ we have $a_i < a_j \Leftrightarrow  b_i < b_j $.  
For $ I = \{ i_1 < i_2 < \cdots < i_k \} \subseteq [n]$ and $\pi \in \Sn$, let $\pi [I]$ denote the permutation in $\Sx{k}$ which is order-isomorphic to the sequence $\pi_{i_1} \pi_{i_2} \cdots \pi_{i_k}$.
A~permutation $\pi \in \Sn$ \textsl{contains} a~permutation $\sigma \in \Sx{k}$ if there exists an~$I$ such that $\pi [I] = \sigma$.
We write $\sigma \preceq \pi$ to denote this.
If $\pi$ does not contain $\sigma$ we say that $\pi$ \textsl{avoids} $\sigma$.
In this context we usually call~$\sigma$ \textsl{a pattern}.
Similarly, for a set of patterns $\Pi$ we say that a~permutation $\tau$ is $\Pi$-avoiding if it is $\sigma$-avoiding for every $\sigma \in \Pi$. 
For a pattern $\sigma$ let $\Sn(\sigma)$ be the set of all $\sigma$-avoiding permutations of length~$n$, and $S_n(\sigma)$ its cardinality.
More generally for a set of permutations $\Pi$, let $\Sn(\Pi)$ denote the set of all $\Pi$-avoiding permutations of length~$n$, and $S_n(\Pi)$ its cardinality.

The \textsl{descent set} of $\sigma \in \Sn$ is the set $\Desc(\sigma) = \{ i \mid \sigma_i > \sigma_{i+1}\}$ and the \textsl{major index} is the sum of all its members $maj(\sigma) = \sum\nolimits_{i \in \Desc(\sigma)} i$.
We will consider the distribution of major index over pattern-avoiding permutations.

\begin{definition}
\label{def:Mnx}
Let $\Mx{n}{m}{\sigma}$ denote the set of all $\sigma$-avoiding permutations of length~$n$ with major index~$m$, and $M_n^m(\sigma)$ its cardinality.
Similarly let $\Mx{n}{m}{\Pi}$ be the set of all permutations from $\Sn(\Pi)$ with major index $m$, and $M^m_n(\Pi)$ its cardinality.
For an example of the values $M^m_n(\sigma)$ for a specific pattern, see Table~\ref{table:1324perm}.
\end{definition}

\begin{table}[h]
\centering
\begin{tabular}{ c c c c c c c c c c c c c c c } 
   1&    &    &    &    &    &    &    &    &    &    &    &    & \\
   1&   1&    &    &    &    &    &    &    &    &    &    &    & \\
   1&   2&   2&   1&    &    &    &    &    &    &    &    &    &  \\
   1&   3&   4&   6&   5&   3&   1&    &    &    &    &    &    & \\
   1&   4&   6&  12&  16&  19&  16&  15&   9&   4&   1&    &    &  \\
   1&   5&   8&  19&  29&  45&  58&  65&  73&  65&  57&  39&  29&  \ldots\\
   1&   6&  10&  27&  44&  76& 119& 164& 212& 260& 287& 299& 303& \ldots \\
\end{tabular}
\caption{The number of $1324$-avoiding permutations with a fixed major index.
The $m$th entry in the $n$th row is the value $M^m_n(1324)$, with $n$ starting at 1 and $m$ starting at $0$.}
\label{table:1324perm}
\end{table}

Now we will introduce a decomposition which will later prove to be very useful.
Let $\mathbb{N}^d_0$ be the set of $d$-tuples of non-negative integers and for every $a \in \mathbb{N}^d_0$ define its size $|a|=\sum^{d}_{i=1}a_i$.
We will decompose an arbitrary permutation into a smaller permutation and a tuple of non-negative integers.
Let $\pi \in \Sn$ be a permutation and $k$ a natural number, such that the sequence $\pi_{k+1} \cdots \pi_n$ is strictly increasing.
Then we can store the structure of such permutation in a shorter permutation~$\sigma$ order-isomorphic to $\pi_1 \cdots \pi_k$, and a $(k+1)$-tuple which describes the vertical gaps between the letters $\pi_1 \cdots \pi_k$.

\begin{definition}
\label{def:decomposition}
Let $\pi \in \Sn$ be a permutation and $k \in [n]$ such that the sequence $\pi_{k+1} \cdots \pi_n$ is strictly increasing.
Let $\sigma$ be the permutation order-isomorphic to the sequence $\pi_1 \cdots \pi_k$ and $a \in \mathbb{N}^{k+1}_0$ the only $(k+1)$-tuple of size $|a| = n - k$ such that $\pi_i = \sigma_i + \sum^{\sigma_i}_{j=1}a_j$ holds for every $i \in [k]$.
Then we say that $\pi$ can be \textsl{decomposed} into $\sigma$ and $a$, denoted by $\pi = \sigma \cdot a$.
\end{definition}

We can also look at the decomposition from the other side as an operation, which increases the vertical gaps between the letters of $\sigma$ and then fills them with increasing suffix.
See Figure~\ref{fig:decomposition}.

\begin{figure}[h!]
\centering
\includegraphics[scale=\figurescale]{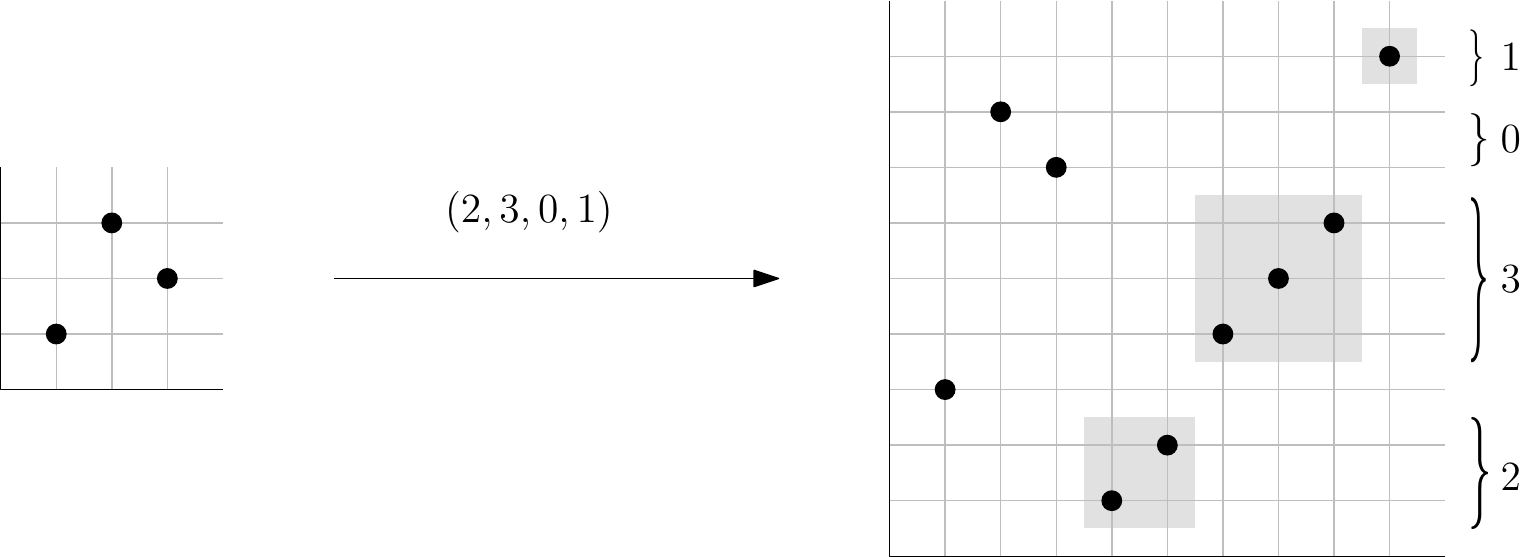}
\caption{For a permutation $\sigma = 132$ and 4-tuple $a = (2,3,0,1)$ we have $\pi = \sigma \cdot a = 387124569$.}
\label{fig:decomposition}
\end{figure}

\begin{definition}
\label{def:core}
For a permutation $\pi$ that can be expressed as $\pi = \gamma \cdot a$ for some $\gamma \in \Sx{k}$ and  $a \in \mathbb{N}^{k+1}_0$, we call $\gamma$ \textsl{the core of~$\pi$} and~$a$ \textsl{the padding profile of~$\pi$} if~$k$ is the last descent of~$\pi$.
In other words, $\pi = \gamma \cdot a$ is a decomposition into a core and a padding profile of~$\pi$ if there is $i \leq \gamma_k$ such that $a_i > 0$.
For~$\pi = 12 \cdots n$, the core of~$\pi$ is the empty permutation and its padding profile is $a \in \mathbb{N}_0$ equal to the length of~$\pi$.
\end{definition} 

Observe that the major index of a permutation~$\pi$ is determined only by its core.
Therefore, let us define the following statistic which characterizes the cores of permutations with a given major index.

\begin{definition}
\label{def:maj+}
For a permutation $\pi$, let the \textsl{extended major index of~$\pi$}, denoted by $maj^+(\pi)$, be the sum of its major index and its length, i.e.,
$$maj^+(\pi) = |\pi| + maj(\pi).$$
\end{definition}

For every permutation~$\pi$ with a core~$\gamma$, we have $maj(\pi) = maj^+(\gamma)$.
Notice also that for any $\pi$, if $\pi$ contains $\sigma$ then $maj^+(\pi)\geq maj^+(\sigma)$.

\section{Monotonicity of columns}
\label{sec:fixed-major}

In this section, we will focus on the distribution of major index over permutations avoiding a single pattern.
Observe that each column of Table~\ref{table:1324perm} is weakly increasing from top to bottom.
In other words, for a fixed major index~$m$ the number of $1324$-avoiding permutations of length $n+1$ is at least the number of $1324$-avoiding permutations of length~$n$.
We will show that this claim holds in general for any single pattern~$\sigma$ except for the increasing patterns (i.e., the patterns of the form~$12 \cdots k$).

%\section{Inserting an element}

First let us define a simple operation of inserting an element into a permutation.
Later we will prove two elementary properties of this operation.

\begin{definition}
\label{def:insert}
For a permutation $\pi \in \Sn$ and $k,l \in [n+1]$, let $\pi[k \to l] \in \Sx{n+1}$ be a permutation created by inserting the letter $l$ at the $k$-th position. In other words $\pi[k \to l]$ is the permutation order-isomorphic to the sequence $\pi_1 \cdots \pi_{k-1} \left(l - \frac{1}{2}\right) \pi_{k} \cdots \pi_{n}$.
See Figure~\ref{fig:insertion}.
\end{definition} 

\begin{figure}[!h]
\centering
\includegraphics[scale=\figurescale]{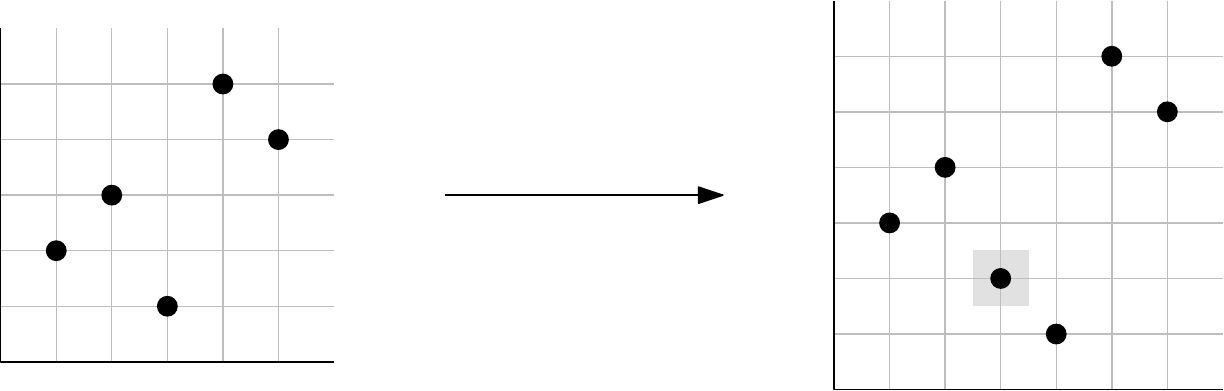}
\caption{Example of an insertion $23154[3 \to 2] = 342165$}
\label{fig:insertion}
\end{figure}

\begin{lemma}
\label{lemma:patterninsert}
Let $n \in \mathbb{N}$, $k, l \in [n + 1]$ and $\pi \in \Sn(\sigma)$. If there is $I$ such that $\pi[k \to l][I] = \sigma$, then $k \in I$.
\end{lemma}
\begin{proof}
Suppose that we have $I = \{i_1 < \cdots < i_m\}$ such that $k \not \in I$.
Let $J = \{j_1 < \cdots < j_m\}$ be a subset of $[n]$ defined by
\[j_t = 
\begin{cases} 
i_t     \quad&\mbox{if } i_t < k \\ 
i_t - 1 &\mbox{if } i_t > k \, .
\end{cases}\]
Since $\pi[k \to l]$ restricted to indices other than $k$ is order-isomorphic to $\pi$ we obtain $\sigma = \pi[k \to l][I] = \pi[J]$, contradicting the fact that $\pi$ avoids $\sigma$.
\end{proof}

\begin{lemma}
\label{lemma:desciso}
Let $n \in \mathbb{N}$, $k, l \in [n]$ and $\pi \in \Sn$. If $\Desc(\pi) \subseteq [k - 1]$, $l \leq \pi_k$ and either $k = 1$ or the sequences $\pi_{k-1} (l - \frac{1}{2})$ and $\pi_{k-1} \pi_{k}$ are order-isomorphic, then $\Desc(\pi) =\nobreak \Desc(\pi[k \to l])$.
\end{lemma}
\begin{proof}
As before $\pi[k \to l]$ restricted to indices other than $k$ is order-isomorphic to $\pi$.
Therefore for every index $i < k - 1$, $i \in \Desc(\pi)$ if and only if $i \in \Desc(\pi[k \to l])$.
And since we know that all the elements of $\Desc(\pi)$ are smaller than~$k$, we get $\Desc(\pi[k \to l]) \subseteq [k]$.
We are left with the two indices $k$ and $k-1$.
Observe that $k \not \in \Desc(\pi[k \to l])$ because $l \leq \pi_k$.
And from the last condition we obtain $k - 1 \in \Desc(\pi)$ if and only if $k - 1 \in \Desc(\pi[k \to l])$.
\end{proof}

%Now we can prove the main result of this chapter about the permutation classes avoiding a single pattern.

\begin{theorem}
\label{thm:monotonity}
For every $n, m, k \in \mathbb{N} $ and $\sigma \in \Sx{k}$ with $\Desc(\sigma) \neq \emptyset$ we have the inequality $M_n^m (\sigma) \le M_{n+1}^m(\sigma)$.
\end{theorem}

\begin{proof}
To prove this theorem we will construct an injective mapping~$f$ from $\Mx{n}{m}{\sigma}$ to $\Mx{n+1}{m}{\sigma}$.
In order to find an image for $\pi \in \Mx{n}{m}{\sigma}$ we introduce the following permutation statistics.

\begin{definition}
\label{def:tail}
For $\sigma \in \Sn$ let $tail(\sigma)$ denote the largest $i$ such that $\sigma_{n + 1 -i}$ $\sigma_{n + 2 -i} \cdots \sigma_n$ are all fixed points. 
And similarly let $slope(\sigma)$  be the largest $i$ such that the sequence $\sigma_{n + 1 -i}$ $\sigma_{n + 2 -i} \cdots \sigma_n$ is strictly increasing.
Recalling Definition~\ref{def:core}, we see that $slope(\sigma)$ is the size of the padding profile of~$\sigma$ and $tail(\sigma)$ is the value of its last coordinate.
See Figure~\ref{fig:tail}. 
\end{definition} 

\begin{figure}[h!]
\centering
\includegraphics[scale=\figurescale]{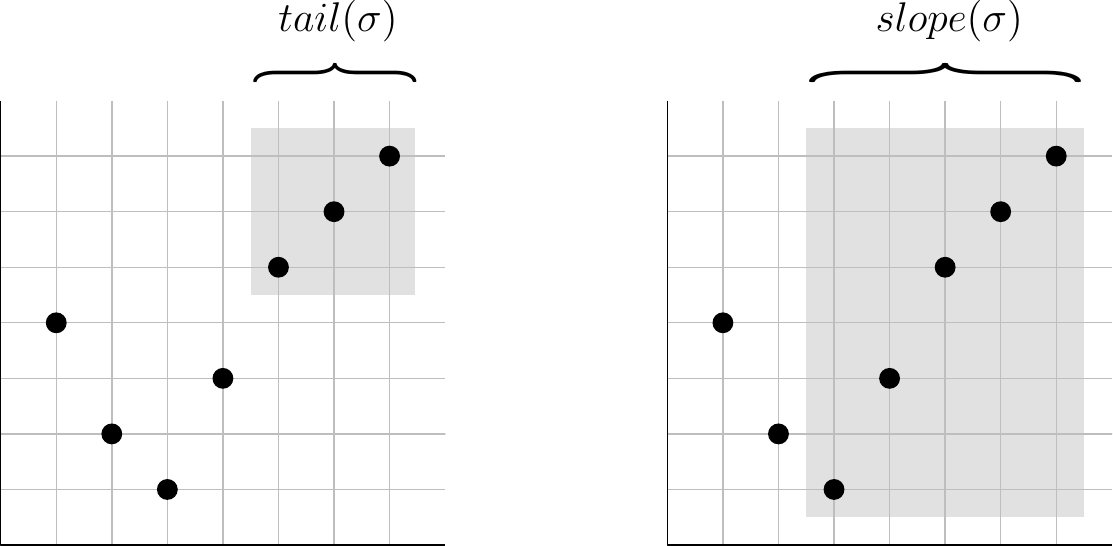}
\caption{The tail and slope statistics of the permutation $\sigma = 4213567$}
\label{fig:tail}
\end{figure}

\begin{case}
\label{monocs1}
First we solve the easy case where $tail(\sigma) = 0$.
We simply extend $\pi$ by inserting the letter $n + 1$ at the end, i.e.,
\[f(\pi) = \pi[n + 1 \to n + 1].\]
It is clear that $f$ preserves the descent set, which implies $maj(\pi) = maj(f(\pi))$.
Now suppose there is $I = \{i_1 < \cdots < i_k\}$ such that $f(\pi)[I] = \sigma$.
Lemma~\ref{lemma:patterninsert}  implies $i_k = n + 1$.
But that would lead to $\sigma_k = f(\pi)[I]_k = k$ which contradicts the assumption that $tail(\sigma) = 0$.
\end{case}

\begin{case}
\label{monocs2}
Suppose now that $tail(\sigma) \neq 0$ and $slope(\pi) \geq tail(\sigma)$.
Then we create the image of $\pi$ by expanding the element at the position $n + 1 -tail(\sigma)$ into two.
See Figure~\ref{fig:case2}.
\[f(\pi) = \pi[t \to \pi_t] \quad \text{where $t = n + 1 - tail(\sigma)$} .\]

\begin{figure}[h!]
\centering
\includegraphics[scale=\figurescale]{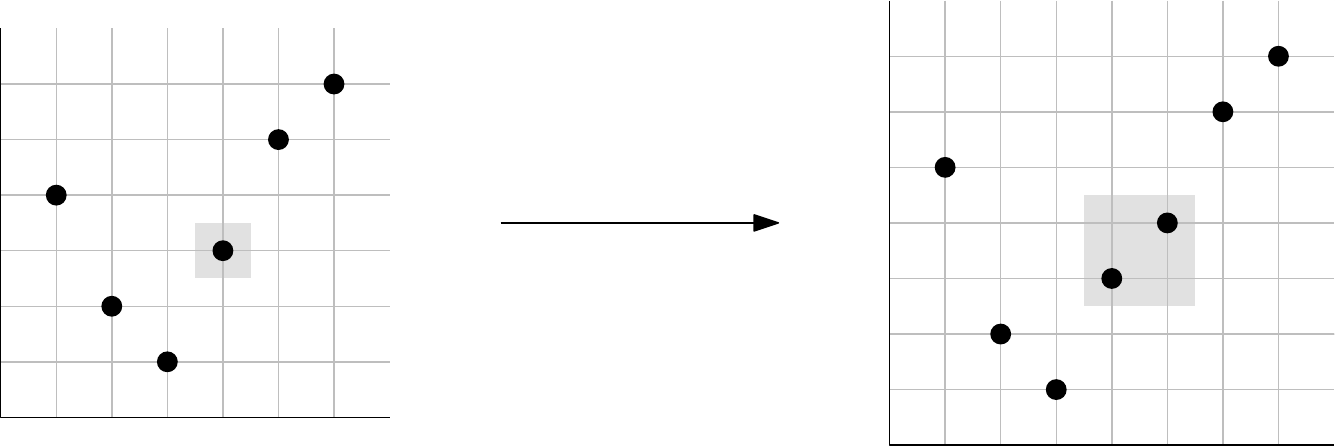}
\caption{Example of a construction from Case~\ref{monocs2}. 
Consider a permutation $\pi = 421356$, which has $slope(\pi) = 4$, and suppose we have a pattern $\sigma$ with $tail(\sigma) = 3$.
Then $f(\pi) = 5213467$.
}
\label{fig:case2}
\end{figure}

Because all the conditions from Lemma~\ref{lemma:desciso} are met, we get $\Desc(\pi) = \Desc(f(\pi))$ which implies $maj(\pi) = maj(f(\pi))$.

Next we want to show that $f(\pi)$ avoids~$\sigma$.
Suppose there is $I = \{i_1 < \cdots < i_k\}$ such that $f(\pi)[I] = \sigma$. 
Again from Lemma~\ref{lemma:patterninsert} we obtain $t = i_j \in I$ for some $j$.
Observe that since there are only $tail(\sigma)$ indices in $f(\pi)$ larger than~$t$, we get a lower bound $j \geq k - tail(\sigma)$.

Now we will use different arguments depending on whether this holds as an equality or not.
First suppose that $j = k - tail(\sigma)$.
This means that~$I$ also contains all the indices larger than~$j$, in particular $i_{j+1} = t  + 1$.
Following Definition~\ref{def:tail}, $j$ is then the largest index such that $\sigma_j \neq j$, implying $\sigma_j < j$.
This means there is a letter $\sigma_l$ to the left of $\sigma_j$ such that $\sigma_l > \sigma_j$ and all the letters to the right of $\sigma_j$ are larger than $\sigma_l$.
Therefore, looking at the indices $j$, $j+1$ and $l$ we have $\sigma_j < \sigma_l < \sigma_{j+1}$ and the same inequality goes for $f(\pi)[I]$.
Translated to the indices of $f(\pi)$ the inequality $f(\pi)_{t} < f(\pi)_{i_l} < f(\pi)_{t+1}$ must hold.
But recalling the definition of $f(\pi)$ there is no index $p$ such that $f(\pi)_{t} < f(\pi)_{p} < f(\pi)_{t+1} $. 

Suppose now that $j > k - tail(\sigma)$.
In this case there must be $l > j$ such that $l \not \in I$.
We aim to show that $J = \{j_1 < \cdots < j_k\} = I \setminus \{t\} \cup \{l\}$ satisfies $f(\pi)[J] = f(\pi)[I] = \sigma$, which would lead to a contradiction since~$\pi$ contains $f(\pi)[J]$. 

We know that  $\sigma_j \sigma_{j+1} \cdots \sigma_k$ are all fixed points following Definition~\ref{def:tail}.
Observe that for every index $p$ we have the inequality $j_p \geq i_p$ with equality on the indices smaller than~$j$.
Therefore, $f(\pi)[I]$ and $f(\pi)[J]$ restricted to the first $j-1$ letters are order-isomorphic.
The only thing left is to check that the other letters of $f(\pi)[J]$ are fixed points.
The sequence $f(\pi)_t \cdots f(\pi)_{n+1}$ is increasing, thus its subsequence $f(\pi)_{j_j}$ $f(\pi)_{j_{j+1}} \cdots f(\pi)_{j_k}$ is also increasing and moreover $f(\pi)_{j_j} > f(\pi)_{i_j}$.
Then $f[J]_j$ and subsequently all the succeeding letters of $f[J]$ must be fixed points.
Together, this means that indeed $f(\pi)[I] = f(\pi)[J]$. 
\end{case}

\begin{case}
\label{monocs3}
Finally, suppose that $tail(\sigma) \neq 0$ and $slope(\pi) < tail(\sigma)$.
Then we simply insert the letter~$1$ at the rightmost possible position without creating a new descent.
See Figure~\ref{fig:case3}.
\[f(\pi) = \pi[n + 1 - slope(\pi) \to 1].\]

\begin{figure}[h!]
\centering
\includegraphics[scale=\figurescale]{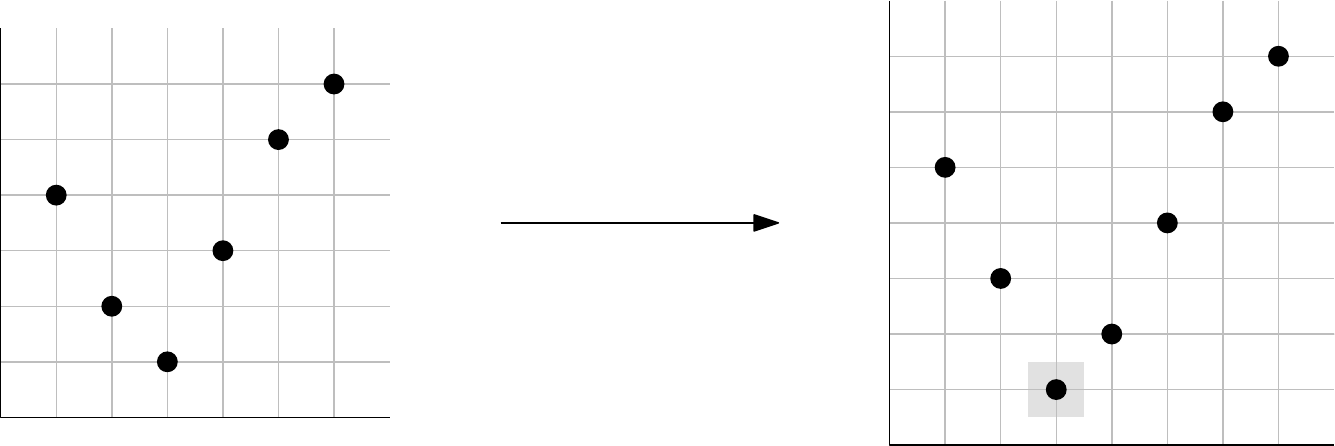}
\caption{Example of a construction from Case~\ref{monocs3}. 
Consider a permutation $\pi = 421356$, which has $slope(\pi) = 4$, and suppose we have a pattern $\sigma$ with $tail(\sigma) > 4$.
Then $f(\pi) = 5312467$.
}
\label{fig:case3}
\end{figure}

As before, we obtain $maj(\pi) = maj(f(\pi))$ from Lemma~\ref{lemma:desciso}.
If there is $I = \{i_1 < \cdots < i_k\}$ such that $f(\pi)[I] = \sigma$, then Lemma~\ref{lemma:patterninsert} implies $n + 1 - slope(\pi) = i_j$ for some~$j$.
The $j$-th letter of~$\sigma$ must be its minimum since $f(\pi)_{i_j} = 1$ is the minimum of $f(\pi)$.
On the other hand, because $n + 1 - slope(\pi) > n + 1 - tail(\sigma)$ and $\Desc(\sigma) \neq \emptyset$, there must be~$q$ such that $\sigma_q < \sigma_j$, which yields a contradiction.
\end{case}

The only remaining part is to show that~$f$ is injective.
Suppose there are $\pi_1 \neq \pi_2$ such that $f(\pi_1) = f(\pi_2)$.
From the properties of $f(\pi_1)$ we can tell unambiguously whether it was obtained through Case~\ref{monocs1},~\ref{monocs2} or~\ref{monocs3}.
And following the definitions of~$f$ in these particular cases it is clear that necessarily $\pi_1 = \pi_2$. 
\end{proof}

In Theorem~\ref{thm:monotonity}, the assumption $\Desc(\sigma)\neq\emptyset$ is necessary, because in the case of a pattern $\sigma = 12 \cdots k$ and fixed $m \in \mathbb{N}$ there is $n_0 \in \mathbb{N}$ such that for every~$n$ larger than~$n_0$ we have $M_n^m (\sigma) = 0$.
This follows directly from the Erdős–Szekeres theorem~\cite{Erdos1935}, which states that any permutation of length $n > m(k-1)+1$ contains either the increasing pattern of length~$k$ or the decreasing pattern of length $m + 1$, forcing the major index to be larger than~$m$.

Applying a similar argument as in the proof of Theorem~\ref{thm:monotonity}, we could show that $M_n^m(\Pi) \leq M_{n+1}^m(\Pi)$ for  any set of patterns with the same tail which does not contain any increasing pattern.
One could think that indeed for any set of patterns the columns are either eventually zero or weakly increasing.
But this is not true even for small sets of short patterns.
For example, consider a set $\Pi = \{3412, 1324\}$ of just two patterns.
In this case $M^5_6(\Pi) = M^5_8(\Pi) = 21$ and $M^5_7(\Pi) = 20$ (see Table~\ref{table:badtable}).
But we can easily show that $M^m_n(\Pi)$ tend to infinity for $m \geq 3$.
Let $\pi(n) = 12 \cdots (n-2) [m-1 \to 1]$ and $\pi(n,k) = \pi(n)[1 \to k]$, then $\pi(n,k) \in M^m_{n}(\Pi)$ for $n \geq m$ and $k > 2$. 
Therefore, $M^m_n(\Pi)$ tends to infinity because there are linearly many $\pi(n,k)$ for a fixed~$n$.

\begin{table}[h!]
\centering
\begin{tabular}{ c c c c c c c c c c c c c c c } 
    &    &    &    &    &    &    &    &    &    &    &    &    &    \\
   1&    &    &    &    &    &    &    &    &    &    &    &    &    \\
   1&   1&    &    &    &    &    &    &    &    &    &    &    &    \\
   1&   2&   2&   1&    &    &    &    &    &    &    &    &    &    \\
   1&   3&   3&   6&   5&   3&   1&    &    &    &    &    &    &    \\
   1&   4&   3&   9&  12&  16&  12&  15&   9&   4&   1&    &    &    \\
   1&   5&   3&  13&  12&  \cellcolor[HTML]{E1E1E1} 21&  38&  31&  48&  41&  44&  29&  29&\ldots\\
   1&   6&   3&  18&  13&  \cellcolor[HTML]{E1E1E1} 20&  49&  62&  63& 105&  95& 109& 162&\ldots\\
   1&   7&   3&  24&  14&  \cellcolor[HTML]{E1E1E1} 21&  62&  62& 105& 105& 221& 169& 222& \ldots\\
\end{tabular}
\caption{The number of permutations avoiding $\Pi = \{3412, 1324\}$ with a fixed major index.
The $m$th entry in the $n$th row is the value $M^m_n(\Pi)$, with $n$ starting at 1 and $m$ starting at $0$.
The problematic values are highlighted.}
\label{table:badtable}
\end{table}

\section{Asymptotic behaviour}
\label{sec:asymptotics}

We have seen that for most single patterns the inequality $M^m_n(\sigma) \leq M^m_{n+1}(\sigma)$ holds (recall Theorem~\ref{thm:monotonity}).
Let us now focus on the asymptotic behaviour of $M^m_n(\sigma)$ for a fixed~$m$ as~$n$ tends to infinity.
More generally, we are interested in the asymptotic behaviour of $M^m_n(\Pi)$ for a (possibly infinite) set of permutations~$\Pi$.
Our first goal is to show that for a fixed~$m$ and arbitrary $\Pi$, $M_n^m(\Pi)$ is eventually equal to a polynomial.

%\section{Polynomial growth}

By recalling Definition~\ref{def:core}, observe that a permutation is uniquely determined by its core and its padding profile while its major index is determined only by the core.
Furthermore, for any permutation $\tau \in \Mx{n}{m}{\Pi}$ all the elements of $\Desc(\tau)$ are smaller than $m+1$, thus making the core of any such permutation shorter than $m+1$.
This means that all the permutations with major index~$m$ have only a finite number of unique cores.

\begin{definition}
\label{def:setofcores}
Let $C(m, \Pi)$ denote the finite set of all the distinct cores of permutations from $\M ^m (\Pi)$, where $\M ^m (\Pi)$ the set of all $\Pi$-avoiding permutations with major index~$m$, i.e. $\M^m (\Pi) = \bigcup_{n \geq 1} \M^m_n (\Pi)$.
\end{definition}

Note that every core $\gamma \in C(m, \Pi)$ satisfies $maj^+(\gamma) = m$ (recall Definition~\ref{def:maj+}).
Therefore the permutation statistic which assigns to each permutation its core, is in fact a refinement of the major index.

\begin{definition}
\label{def:fixedcore}
For $\gamma\in C(m,\Pi)$, let $\M^{[\gamma]}_n(\Pi)$ be the set of permutations from $\M^m_n(\Pi)$ which have the core $\gamma$, and let $M^{[\gamma]}_n(\Pi)$ be its cardinality.
\end{definition}

Clearly $M^m_n(\Pi) = \sum_{\gamma \in C(m, \Pi)} M^{[\gamma]}_n(\Pi)$.
This means that in order to prove the polynomial behaviour of $M^m_n(\Pi)$ for a fixed~$m$, it is enough to prove the polynomial behaviour of $M^{[\gamma]}_n(\Pi)$ for a fixed core~$\gamma$.
And because the decomposition of a permutation into its core and its padding profile is unique, we can enumerate $\M^{[\gamma]}_n (\Pi)$ by counting all the possible padding profiles.

\begin{lemma}
\label{lemma:polynomialcore}
Let~$\Pi$ be any set of permutations and $\gamma \in \Sx{k}$ a permutation.
Then there exists a polynomial~$P$ and an integer~$n_0$ such that for every $n \geq n_0$, $M^{[\gamma]}_n(\Pi) = P(n)$.
\end{lemma}

\begin{proof}
We will use a known property of down-sets of integer compositions.
Define a partial order $\leq$ on $\mathbb{N}^d_0$ as $(a_1, \ldots, a_d)\leq (b_1, \ldots, b_d)$ if for every $i \in [d]$ we have $a_i \leq b_i$.
A set $A \subseteq \mathbb{N}^d_0$ is a \textsl{down-set} in $\mathbb{N}^d_0$ if for every $a \in A$ and $b \leq a$, $b$ belongs to $A$ as well.
Unfortunately the set of all padding profiles from $\M^{[\gamma]}_n (\Pi)$ is not a down-set, but we can express it as a difference of two down-sets.
Define the following sets 
\begin{alignat*}{4}
&A_n &&= \{ a \mid a \in \mathbb{N}^{k+1}_0 \wedge \gamma \cdot a \in \Sn(\Pi) \} \quad &\mbox{and} \quad &A = \bigcup\nolimits_{n \geq 0} A_n \\
&B_n &&= \{ a \mid a \in A_n \wedge \forall i \leq \gamma_k \colon a_i = 0 \} \quad &\mbox{and} \quad &B = \bigcup\nolimits_{n \geq 0} B_n .
\end{alignat*}
 
Let us check that both $A$ and $B$ are down-sets in $\mathbb{N}^{k+1}_0$.
If $a$ belongs to $A$ and $b \leq a$, then the permutation $\gamma \cdot a$ contains the permutation $\gamma \cdot b$ and therefore $\gamma \cdot b$ must be $\Pi$-avoiding and $b$ belongs to $A$.
To show that $B$ is down-set, consider $a \in B$ and $b \leq a$.
We know from previous argument that $b$ also belongs to $A$ and the second condition holds since for every $i \in [\gamma_k]$ we have $b_i \leq a_i = 0$ implying $b_i = 0$.

The padding profiles of permutations from $\M^{[\gamma]}_n (\Pi)$ have at least one of the first $\gamma_k$ values positive, because such permutation has a descent at the $k$-th position.
But these are exactly the tuples which belong to $A_n$ but not to $B_n$.
Since $B_n$ is a subset of $A_n$ we get $M^{[\gamma]}_n (\Pi) = |A_n| - |B_n|$.
To complete the proof, we will use the following fact due to Stanley \cite{Stanley1975, Stanley1976}.

\begin{proposition}[Stanley]
Let $d$ be a positive integer and let $S$ be a down-set in $\mathbb{N}^d_0$.
Let $H(n)$ be the number of elements of $S$ with size~$n$.
Then there exists a polynomial~$P$ and an integer~$n_0$ such that for every $n \geq n_0$, $H(n) = P(n)$.
\end{proposition}

From this fact, we obtain that $|A_n|$ and $|B_n|$ are both polynomials for sufficiently large~$n$, therefore $M^{[\gamma]}_n(\Pi)$ is eventually equal to a polynomial as well.
\end{proof}

\begin{corollary}
\label{cor:polynomialmajor}
For a set of permutations~$\Pi$ and $m \in \mathbb{N}_0$, there exists a polynomial~$P$ and an integer~$n_0$ such that for every $n \geq n_0$, $M^m_n(\Pi) = P(n)$.
\end{corollary}

Since we now know that the numbers $M_n^m(\Pi)$ are eventually equal to a polynomial, we can introduce the following notation.

\begin{definition}
\label{def:degree}
For a set of permutations~$\Pi$, let $deg(m, \Pi)$ be the degree of the polynomial~$P$ such that $M^m_n(\Pi) = P(n)$ for~$n$ large enough.
For a zero polynomial~$P$, let $deg(m, \Pi) = 0$.
\end{definition} 

Observe that for an arbitrary set of permutations~$\Pi$ and $\Omega \subseteq \Pi$, it follows that $deg(m, \Pi) \leq deg(m, \Omega)$.
This holds since any $\Pi$-avoiding permutation is trivially $\Omega$-avoiding too.

Now we would like to know how these degrees depend on~$m$ and on the structure of permutations in~$\Pi$.
It turns out that there is one important statistic of patterns which affects the degree $deg(m, \Pi)$.

\begin{definition}
\label{def:magnitude}
For a permutation~$\pi$ we will define \textsl{the magnitude of~$\pi$} as
\[
mg(\pi) =
\begin{cases}
0 & \text{if }\Desc(\pi)=\emptyset\\
k & \text{if }\Desc(\pi)=\left\{k\right\}\\
+\infty & \text{if }|\Desc(\pi)|\ge 2.
\end{cases}
\]
For a set of permutations~$\Pi$ \textsl{the magnitude of~$\Pi$}, denoted by $mg(\Pi)$, is the minimal magnitude of a permutation $\sigma \in \Pi$.
For the empty set of patterns, $mg(\emptyset) = +\infty$.
\end{definition}

Let us make an important observation about magnitude.
If a permutation~$\pi$ contains a pattern~$\sigma$ then necessarily $mg(\pi) \geq mg(\sigma)$. 

%\subsection{Sets of infinite magnitude}

As we will show, the magnitude of~$\Pi$ plays a key role in determining the value of $deg(m,\Pi)$. 
To prove this, let us first focus on the sets~$\Pi$ of infinite magnitude.
In this particular case, we can also determine the leading coefficient of the polynomial $M^m_n(\Pi)$, which will prove to be useful later.

\begin{proposition}
\label{prop:infinitemg}
Let~$\Pi$ be a set of permutations with $mg(\Pi) = +\infty$.
Then $deg(m, \Pi) = m$ and $M^m_n(\Pi) = \frac{n^m}{m!} + O(n^{m-1})$  as $n\to\infty$.
\end{proposition}

\begin{proof}
First observe that for $m = 0$ the proposition simply states that $M^0_n(\Pi) = 1$ for $n \geq n_0$.
But that is clear since $\M^0_n(\Pi) = \M^0_n(\emptyset) = \left\{12 \cdots n \right\}$.
Therefore, in the rest of the proof suppose that $m \geq 1$. 

\noindent
In order to prove our proposition, it is sufficient to prove the following claims.
\begin{enumerate}[noitemsep]
  \item For the core $\epsilon = 12 \cdots m$ we have $M^{[\epsilon]}_n(\Pi) = \frac{n^m}{m!} + O(n^{m-1})$.
  \item For every $\gamma \in C(m, \Pi) \setminus \left\{\epsilon\right\}$ there is a constant $\beta = \beta(\gamma, m, \Pi)$ such that $M^{[\gamma]}_n(\Pi) \leq \beta n^{m-1}$. 
\end{enumerate}

To prove the first claim, simply observe that any permutation with the core~$\epsilon$ has a finite magnitude, which makes it $\Pi$-avoiding.
By choosing the first~$m$ letters we uniquely get every permutation with the core~$\epsilon$ plus the permutation $12\cdots n$.
That gives us the desired enumeration $M^{[\epsilon]}_n(\Pi) = \binom{n}{m} - 1 = \frac{n^m}{m!} + O(n^{m-1}) $.

To prove the second claim, fix a core $\gamma \in C(m, \Pi) \setminus \left\{\epsilon\right\}$ of length $k$.
First observe that for $\gamma \neq \epsilon$ necessarily $k \leq m-1$.
We will bound $M^{[\gamma]}_n(\Pi)$ from above by the number of all the permutations of length~$n$ which can be expressed as $\gamma \cdot a$ for some tuple $a$.
This yields the upper bound $M^{[\gamma]}_n(\Pi) \leq \binom{n}{k} \leq \binom {n}{m-1} \leq \beta n^{m-1}$ for some $\beta$.

These claims together with Corollary~\ref{cor:polynomialmajor} give the desired polynomial behaviour.
\end{proof}

%\subsection{Sets of finite magnitude}

Let us now focus on the problem of determining $deg(m,\Pi)$ for a set~$\Pi$ of finite magnitude.
As we will show in this section, the asymptotic behaviour of these sets is far more complicated than that of sets with infinite magnitude.
Our main result is providing the values $deg(m, \Pi)$ as a function of~$m$.
As in Proposition~\ref{prop:infinitemg}, we will construct a suitable core and bound $deg(m, \Pi)$ from below by counting all the possible padding profiles.
On the other hand, we will use a different approach for obtaining the upper bound.
For a fixed core~$\gamma$, we will bound $M^{[\gamma]}_n(\Pi)$ in terms of how many coordinates of a padding profile~$a$ can be large if $\gamma \cdot a$ avoids~$\Pi$.

\begin{lemma}
\label{lemma:finiteupper}
Let~$\Pi$ be a finite set of permutations and let~$m$ and~$l$ be non-negative integers.
If every permutation~$\pi$ with $maj^+(\pi) \leq m$ and length $|\pi| > l$ contains a core of some permutation in~$\Pi$, then $deg(m, \Pi) \leq l$.
\end{lemma}

\begin{proof}
Let $k$ be the length of the longest permutation in~$\Pi$.
We will prove the claim by showing that for every $\gamma \in C(m,\Pi)$ there is a constant $\alpha = \alpha(\gamma, m, \Pi)$ such that $M^{[\gamma]}_n(\Pi) \leq \alpha n^l$.

Fix a core~$\gamma \in 	C(m, \Pi)$.
For a padding profile $a \in \mathbb{N}^{d}_0$ we will say that its coordinate $a_i$ is \textsl{bad} if $a_i > k$.
We claim that every permutation in $M^{[\gamma]}_n(\Pi)$ has a padding profile with at most $l+1$ bad coordinates.
Suppose for a contradiction that there is a permutation $\pi \in \M^{[\gamma]}_n(\Pi)$ with at least $l+2$ bad coordinates in its padding profile.
Let $\psi$ be the permutation order-isomorphic to $l+1$ elements from the core of $\pi$ which separate the $l+2$ bad coordinates.
Because $\psi$ is contained in the core it satisfies $maj^+(\psi) \leq m$.
But since it has length greater than~$l$ it must contain a core~$\kappa$ of some permutation~$\sigma \in \Pi$.
Furthermore, let $p \in \mathbb{N}^{l+2}_0$ be the tuple of only the $l+2$ bad coordinates from the padding profile of~$\pi$.
Observe that since $\psi$ contains $\kappa$ and every coordinate of $p$ is larger than $|\sigma|$ then $\psi \cdot p$ must contain $\sigma$.
But that is clearly a contradiction because $\psi \cdot p$ is contained in~$\pi$.

Now it suffices to show that the number of permutations with core~$\gamma$ and at most $l+1$ bad coordinates is smaller than $\alpha n^{l}$ for some $\alpha$.
Let~$d$ be the length of the core~$\gamma$.
First we have $\binom{d+1}{l+1}$ ways to choose the $l+1$ potentially bad coordinates.
We have only constantly many options for the remaining $d-l$ coordinates of the padding profile, which we can bound with $k^{d-l}$.
And finally, we will bound the number of options how to split the remaining elements into the bad coordinates by enumerating the number of ways to split~$n$ elements into~$l+1$ boxes which is $\binom{n+l}{l}$.

This yields the upper bound $M^{[\gamma]}_n(\Pi) \leq k^{d-l} \binom{d+1}{l+1} \binom{n+l}{l}$ and since the only non-constant factor is $\binom{n+l}{l}$, this indeed implies $M^{[\gamma]}_n(\Pi) \leq \alpha n^{l}$ for some $\alpha$.
\end{proof}

By combining Lemma~\ref{lemma:finiteupper} with the Erdős–Szekeres theorem~\cite{Erdos1935}, we obtain a precise characterization of the sets~$\Pi$ for which the degrees $deg(m, \Pi)$ are bounded by a constant independent of~$m$.
This illustrates that sets of patterns containing permutations with both finite and infinite magnitude can behave very miscellaneously.

\begin{proposition}
\label{prop:boundeddegree}
For a set of permutations~$\Pi$, $deg(m, \Pi)$ is bounded by a constant independent of~$m$, if and only if there is $\sigma \in \Pi$ with the core $12 \cdots k$ and $\tau \in \Pi$ with the core $l (l-1) \cdots 1$ for some~$k$ and~$l$.
\end{proposition}

\begin{proof}
To prove one implication, assume that~$\Pi$ contains such~$\sigma$ and~$\tau$.
We know that $deg(m, \Pi) \leq deg(m, \left\lbrace \sigma, \tau \right\rbrace)$.
From the Erdős–Szekeres theorem~\cite{Erdos1935}, it follows that every permutation longer than $(l-1)(k-1)$ contains either $12 \cdots k$ or $l {(l-1)} \cdots 1$.
Therefore, we obtain the inequality $deg(m, \Pi) \leq (k-1)(l-1)$ from Lemma~\ref{lemma:finiteupper}.

We will prove the other implication by proving its contrapositive.
Assume that~$\Pi$ does not contain any permutation with an increasing core.
In other words $mg(\Pi) = + \infty$ and Proposition~\ref{prop:infinitemg} implies that $deg(m,\Pi)=m$.
Therefore, $deg(m, \Pi)$ is unbounded.

Finally, assume that~$\Pi$ does not contain any permutation with a decreasing core.
In this case we cannot precisely express $deg(m,\Pi)$.
However, if $m = \frac{d^2 + d}{2}$ for some integer~$d$ then every permutation with the core $d (d-1) \cdots 1$ is $\Pi$-avoiding and has major index~$m$.
Since there are~$\binom{n-1}{d}$ such permutations of length~$n$, we get the inequality $deg(m, \Pi) \geq d$.
And again $deg(m, \Pi)$ is unbounded.
\end{proof}

Now we will focus on determining the values $deg(m, \Pi)$ for sets~$\Pi$ of finite magnitude.
As we already discussed in Section~\ref{sec:fixed-major}, for any set of permutations~$\Pi$ with magnitude~$0$, we have $M^m_n(\Pi) = 0$ for $n \geq n_0$.
It is not hard to show that the values eventually get constant in the case of sets with magnitude~$1$.

\begin{proposition}
\label{prop:mg1}
If~$\Pi$ is a set of permutations with magnitude $mg(\Pi)=1$, then $deg(m, \Pi) = 0$.
\end{proposition}

\begin{proof}
We know that $deg(m, \Pi) \geq 0$ for every $m$ and $\Pi$.
Therefore, it is sufficient to bound $deg(m, \Pi)$ from above.
Fix a permutation $\tau \in \Pi$ with magnitude~1.
Since $deg(m, \Pi) \leq deg(m, \left\lbrace \tau \right\rbrace)$ and every non-empty permutation contains the pattern~1, we get $deg(m, \Pi) \leq 0$ directly from Lemma~\ref{lemma:finiteupper}.
\end{proof}

The next result determines $deg(m,\Pi)$ for all sets of permutations of magnitude at least~3 where every permutation has a finite magnitude.

\begin{theorem}
\label{thm:magnitude2}
Let~$\Pi$ be a set of permutations such that every permutation $\sigma \in \Pi$ has a finite magnitude and $mg(\Pi) = k$, where $k$ is an integer larger than~2.
Then $deg(0, \Pi) = 0$ and for $m \geq 1$
\[deg(m, \Pi) = \biggl \lfloor \frac{(d-1)(k-1)}{2} + \frac{m}{d} \biggr \rfloor \quad
 \mbox{where} \quad d = \biggl \lceil \frac{1}{2} \biggl(-1 + \sqrt{1 + \frac{8m}{k-1}} \;\biggr) \biggr \rceil .\]
\end{theorem}

\begin{proof}
Any permutation~$\sigma$ with major index~$0$ is strictly increasing, therefore $\sigma$ avoids $\Pi$ and $M^0_n(\Pi) = 1 = n^0$.
In the rest of the proof suppose $m \geq 1$.

\noindent
We will prove the theorem by showing that the following values are equal.
\begin{enumerate}[noitemsep]
  \item The degree of the polynomial $l_1 = deg(m, \Pi)$.
  \item The largest integer~$l_2$ such that there is a $12 \cdots k$-avoiding permutation~$\sigma$ with $maj^+(\sigma) \leq m$.
  \item The largest integer~$l_3$ such that there is a $12 \cdots k$-avoiding permutation~$\pi$ with $maj^+(\pi) = m$.
  \item The value $l_4 = \left\lfloor \frac{(d-1)(k-1)}{2} + \frac{m}{d} \right \rfloor$, where $d = \left \lceil \frac{1}{2} \left(-1 + \sqrt{1 + \frac{8m}{k-1}} \;\right) \right\rceil$.
\end{enumerate}

First observe that trivially $l_2 \geq l_3$.
We will prove $l_3 \geq l_4$ by constructing a $12 \cdots k$-avoiding permutation~$\pi$ of length~$l_4$ satisfying $maj^+(\pi) = m$.
For a permutation $\psi \in \Sn$ we say that $\psi$ is \textsl{co-layered} if~$\psi$ avoids both $132$ and $213$.
Observe that any co-layered permutation is uniquely determined by its descent set.
Let~$d$ be the smallest positive integer such that $\frac{d^2 + d}{2}(k-1) \geq m$.
Furthermore, let~$s$ be the largest integer such that $\frac{d^2 + d}{2}(k-1) - ds \geq m$ and let $p = \frac{d^2 + d}{2}(k-1) - ds - m$.
Note that $s < k - 1$ because otherwise the first inequality would also hold for $d' = d-1$.
We also know that $p < d$ because otherwise the above inequalities would hold for $s' = s+1$ and $p' = p-d$, which contradicts our choice of~$s$.

Let~$\pi$ be a co-layered permutation of length~$d_d$ with the descent set $\Desc(\pi) = \left\{d_1, d_2, \ldots , d_{d-1}\right\}$, where
\[d_i =
\begin{cases}
i (k-1) - s      &\mbox{for } \phantom{d - p} 1 \leq i \leq d-p \\ 
i (k-1) - s - 1  &\mbox{for } \phantom{1} d-p < i \leq d .
\end{cases}\]

\begin{figure}[h]
\centering
\includegraphics[scale=\figurescale]{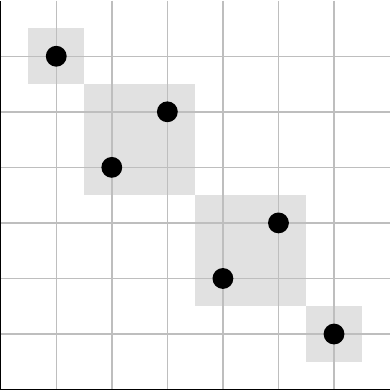}
\caption{
The constructed co-layered permutation for $m = 15$ and $k = 3$. 
In this case we get $d = 4$, $s = 1$ and $p = 1$.
}
\label{fig:layered-core}
\end{figure}

To see that $\pi$ is correctly defined, we will show that the $d_i$ are strictly increasing and positive.
From the inequalities $s < k - 1$ and $p < d$, it follows that $d_1\ge 1$, and that $d_{i+1} > d_i$ for every index~$i$.
Note that this is where the proof would fail for sets with magnitude~$k = 2$, since for $p \neq 0$ we would have $d_{d-p} = d_{d-p+1}$ and we could not construct such permutation.

Observe that~$\pi$ avoids $12 \cdots k$ because the longest increasing subsequence in any co-layered permutation is between two adjacent descents and for any $i \in [d-1]$ we have $d_{i+1} - d_i \leq k-1$.
We see that $\pi$ satisfies
\[maj^+(\pi) = \sum^d_{i=1} d_i = \frac{d^2 + d}{2}(k-1) - ds - p = m .\]

Furthermore, we will show that~$\pi$ has length~$l_4$.
By solving the equations we get
\[ d = \biggl \lceil \frac{1}{2} \biggl(-1 + \sqrt{1 + \frac{8m}{k-1}} \;\biggr) \biggr \rceil,
\quad s = \biggl \lfloor \frac{(d + 1) (k-1)}{2} - \frac{m}{d}\biggr \rfloor
.\]

Notice that we subtract 1 during the calculation of $d_d$ if and only if $p \neq 0$ which happens when~$\frac{(d + 1)(k-1)}{2} - \frac{m}{d}$ is not an integer.
This justifies the following equation
\[d_d = d (k-1) - \biggl \lceil \frac{(d + 1)(k-1)}{2} - \frac{m}{d}\biggr \rceil = \biggl \lfloor \frac{(d-1)(k-1)}{2} + \frac{m}{d} \biggr \rfloor = l_4 .\]

In order to prove $l_4 \geq l_2$, let $\tau$ be a $12 \cdots k$-avoiding permutation with length $t \geq l_4+1$.
Because $\tau$ avoids $12 \cdots k$ there has to be a descent in the sequence $\tau_{t-(k-1)} \cdots \tau_t$, another one in the sequence $\tau_{t-2(k-1)} \cdots \tau_{t-(k-1)}$ and so on.
But that leads to the following inequality
\[
maj^+(\tau) \geq t + \sum_{i = 1}^{i(k-1) < t} \left( t - i(k-1) \right)  > d_d + \sum_{i = 1}^{d-1} d_i = m 
.\]

So far we have proved the equality $l_2 = l_3 = l_4$.
Now we will show that $l_1 \geq l_3$.
Let $\sigma$ be the longest $12 \cdots k$ avoiding permutation with extended major index~$m$.
Observe that any permutation with core $\sigma$ avoids $\Pi$ and has major index~$m$.
We will bound $M^{[\sigma]}_n(\Pi)$ from below with the number of such permutations which have its minimum on the position $l_3+1$.
We can arbitrarily choose~$l_3$ letters, which will form the core, from all letters except the letter~$1$.
That gives us the lower bound $M^m_n(\Pi) \geq M^{[\sigma]}_n(\Pi) \geq \binom{n-1}{l_3} \geq \alpha n^{l_3}$ for some constant $\alpha$.

Finally, we will complete the proof by showing that $l_2 \geq l_1$.
Fix a permutation $\tau \in \Pi$ with the minimal magnitude~$k$.
Trivially the inequality $deg(m, \Pi) \leq deg(m, \left\lbrace \tau \right\rbrace )$ holds.
And because every permutation~$\psi$ with $maj^+(\psi) \leq m$ and length greater than~$l_2$ contains $12 \cdots k$, we get the desired upper bound from Lemma~\ref{lemma:finiteupper}.
\end{proof}

Notice that for $m \leq k - 1$ we obtain from Theorem~\ref{thm:magnitude2} $d = 1$ and $deg(m,\Pi) = deg(m, \emptyset) = m$.
On the other hand for $m \geq k$ the degree is strictly smaller than~$m$ and behaves approximately as $\sqrt{m}$.

As suggested by Proposition~\ref{prop:boundeddegree}, Theorem~\ref{thm:magnitude2} does not hold for the sets~$\Pi$ containing permutations with both finite and infinite magnitude.
Similar claim also cannot hold in general for the sets of magnitude~$2$.
Consider the set $\Pi = \left\{ 132, 231\right\}$ of magnitude~$2$ and $\sigma \in \Sn(\Pi)$.
Let~$j$ be an integer such that $\sigma_j  = 1$.
Then the sequence $\sigma_1 \sigma_2 \cdots \sigma_j$ is decreasing since $\sigma$ avoids 231 and similarly the sequence $\sigma_{j} \sigma_{j+1} \cdots \sigma_{n}$ is increasing because $\sigma$ avoids 132.
In other words, every $\Pi$-avoiding permutation has a decreasing core.
On the other hand, every permutation~$\pi$ with the decreasing core $(d-1)(d-2) \cdots 1$ avoids $\Pi$ and $maj(\pi) = \frac{d^2 + d}{2}$.
As a result, $deg(m, \Pi) \neq 0$ if and only if~$m$ can be expressed as $m = \frac{d^2 + d}{2}$ for some integer~$d$.
Therefore, unlike what we have seen so far, the degrees in this case do not satisfy $deg(m, \Pi) \leq deg(m+1,\Pi)$.

However, we can prove a weaker version of Theorem~\ref{thm:magnitude2} by placing some further restrictions on the set~$\Pi$ of magnitude~$2$.

\begin{proposition}
\label{prop:mg2}
Let~$\Pi$ be a set of permutations such that every permutation $\sigma \in \Pi$ has a finite magnitude and $mg(\Pi) = 2$.
Furthermore, assume that there is $i \in \left\{1,2,3\right\}$ such that every permutation $\pi \in \Pi$ with $mg(\pi) = 2$ has a padding profile $a \in \mathbb{N}^3_0$ with $a_i \neq 0$.
Then $deg(m, \Pi) = \lfloor \frac{-1 + \sqrt{1+8m}}{2} \rfloor$.
\end{proposition}

\begin{proof}
Let~$l$ be the largest integer such that $\frac{l^2 + l}{2} \leq m$. 
By solving the quadratic equation, we get $l = \lfloor \frac{-1 + \sqrt{1+8m}}{2} \rfloor$.
Again to show that the degree of the polynomial is equal to~$l$, we will prove the following claims.
\begin{enumerate}[noitemsep]
  \item There is a constant $\alpha = \alpha(m, k)$ such that $M^m_n(\Pi) \geq \alpha n^l$.
  \item The inequality $deg(m, \Pi) \leq l$ holds.
\end{enumerate}

We will construct a core $\gamma$ for which $M^{[\gamma]}_n(\Pi) \geq \alpha n^l$ holds.
If we have~$m = \frac{l^2 + l}{2}$ we will take as a core the descending permutation of length~$l$.
Every permutation with this core is $\Pi$-avoiding and has major index~$m$, thus giving the desired lower bound.

Otherwise let $d = m - \frac{l^2 + l}{2}$.
Observe that $d \leq l$, because otherwise $\frac{{l'}^2 + l'}{2} \leq m$ would hold for $l' = l + 1$.  
Now we will construct a core of length $l+1$ depending on the $i \in \left\{1,2,3\right\}$, for which the assumptions of the proposition hold.
Let $\epsilon = l \cdots 1$, then we will construct the core~$\gamma$ by inserting one letter to~$\epsilon$, 
\[\gamma =
\begin{cases}
\epsilon[l + 1 - d \to 1]     &\mbox{for } i = 1\\ 
\epsilon[l + 1 - d \to d]     &\mbox{for } i = 2 \\
\epsilon[l + 2 - d \to l+1]   &\mbox{for } i = 3.
\end{cases}\]

For an example see Figure~\ref{fig:core-mg2}.
Observe that $\gamma$ no longer avoids $12$, but it satisfies $maj^+(\gamma) = m$.
Let $a \in \mathbb{N}^{l+2}_0$ be a tuple which satisfies one of the following conditions depending on the value of $i$.
\[
\begin{aligned}
a_1 = 0 \mbox{ , } a_2 \neq 0 \quad &\mbox{for } i = 1\\ 
a_{d+1} = 0 \mbox{ , } a_1 \neq 0    \quad &\mbox{for } i = 2 \\
a_{l+2} = 0 \mbox{ , } a_1 \neq 0    \quad &\mbox{for } i = 3
\end{aligned}\]

\begin{figure}[h]
\centering
\includegraphics[scale=\figurescale]{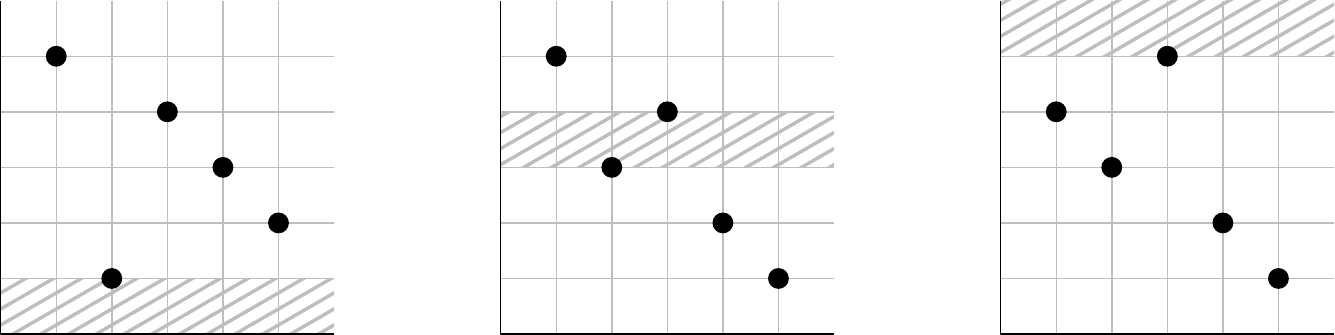}
\caption{
The constructed cores for $m = 15$, from left to right for $i = 1,2,3$.
Any permutation with such core and no element in the grey strip is $\Pi$-avoiding.
}
\label{fig:core-mg2}
\end{figure}

We know that $\gamma \cdot a$ has major index~$m$ and does not contain any permutation with magnitude larger than~2.
But it also cannot contain any permutation $\sigma \in \Pi$ with magnitude~$2$ because of the conditions above.
Therefore $\gamma \cdot a \in \M^m_n(\Pi)$.
Since there are $\binom{n-2}{l}$ such padding profiles, we see that  $M^{[\gamma]}_n(\Pi) \geq \alpha n^l$ for some $\alpha$.

To prove our second claim, fix a permutation~$\tau \in \Pi$ with magnitude~2.
Observe that any permutation~$\sigma$ with length at least $l+1$ for which $maj^+(\sigma) \leq m$ holds, necessarily contains $12$.
Therefore, the upper bound $deg(m, \Pi) \leq l$ is implied by Lemma~\ref{lemma:finiteupper}. 
\end{proof}

As one would expect the formula for $deg(m,\Pi)$ in Proposition~\ref{prop:mg2} gives the same result as the one in Theorem~\ref{thm:magnitude2} for $k=2$.
It is straightforward to check if you express~$m$ as $m = \frac{t^2 + t}{2} + s$ for some integer~$t$ and $s \leq t$.

From Propositions~\ref{prop:infinitemg}, \ref{prop:mg1}, \ref{prop:mg2} and Theorem~\ref{thm:magnitude2} we know the values of $deg(m,\Pi)$ for all sets $\Pi$ with $|\Pi| = 1$ and $mg(\Pi) \geq 1$.
Furthermore, as we already know, for any permutation~$\sigma$ with magnitude~0, we eventually have $M^m_n(\sigma) = 0$.

\begin{corollary}
\label{cor:principalclasses}
For a singleton set $\Pi = \lbrace \sigma \rbrace$  and $k = mg(\Pi) = mg(\sigma)$
\[
deg(m, \Pi) =
\begin{cases}
m &\mbox{ if } k = + \infty \\
0 &\mbox{ if } k \leq 1 \mbox{ or } m = 0\\
s_{k}(m) &\mbox{ otherwise}  
\end{cases}
\]
where $s_k(m) = \biggl \lfloor \frac{(d-1)(k-1)}{2} + \frac{m}{d} \biggr \rfloor$ and $d = \biggl \lceil \frac{1}{2} \biggl(-1 + \sqrt{1 + \frac{8m}{k-1}} \;\biggr) \biggr \rceil$.
Furthermore, $deg(m, \Pi) = m$ for $m < k$ and otherwise $deg(m, \Pi) < m$.
\end{corollary}
\begin{proof}
Since the values of $deg(m, \Pi)$ were determined in the previous claims, we will only prove the inequalities.
It is easier to use bounds on $deg(m, \Pi)$ than to work with the equations for $s_k(m)$.
Let $\epsilon = 12 \cdots m$.
For $m < k$ every permutation with core~$\epsilon$ avoids $\Pi$ and $M^{[\epsilon]}_n(\Pi) = M^{[\epsilon]}_n(\emptyset) \geq \alpha n^m$ for some constant $\alpha$.
Thus $s_k(m) = deg(m, \Pi) = m$.

On the other hand, for $m \geq k$ every permutation~$\pi$ such that $maj^+(\pi) \leq m$ and $|\pi| \geq m$ contains $12 \cdots k$ (in fact~$\epsilon$ is the only such permutation).
And Lemma~\ref{lemma:finiteupper} implies $s_k(m) = deg(m, \Pi) \leq m-1$.
\end{proof}

Moreover, for an arbitrary set of permutations~$\Pi$ we can use Corollary~\ref{cor:principalclasses} to provide an upper bound on $deg(m, \Pi)$.
Let $\tau \in \Pi$ such that $mg(\tau) = mg(\Pi)$, then $deg(m, \Pi) \leq deg(m, \lbrace \tau \rbrace)$.

%\subsection{Asymptotic probability}

Our previous results in this chapter imply a sharp dichotomy for the probability that a random permutation with a fixed major index avoids a specific set of patterns~$\Pi$.

\begin{theorem}
\label{thm:limitprobability}
Let $\Pi$ be a set of permutations and~$m$ a non-negative integer.
Then
\[\lim_{n \to \infty} \frac{M^m_n(\Pi)}{M^m_n(\emptyset)} = 
\begin{cases}
1 \quad &\mbox{if } m < mg(\Pi) \\
0 \quad &\mbox{otherwise.}
\end{cases}\]
\end{theorem}
\begin{proof}
First notice that directly from Proposition~\ref{prop:infinitemg}, it follows that $M^m_n(\emptyset) = \frac{n^m}{m!} + O(n^{m-1})$.
For $m < mg(\Pi)$ every permutation with core $\epsilon = 12 \cdots m$ avoids~$\Pi$.
As we already showed in the proof of Proposition~\ref{prop:infinitemg}, $M^{[\epsilon]}_n(\Pi) = \frac{n^m}{m!} + O(n^{m-1})$.
Therefore, the ratio is approaching~$1$ as~$n$ goes to infinity.

For $m \geq mg(\Pi)$, we know that $deg(m, \Pi) < m$ (recall Corollary~\ref{cor:principalclasses}).
Therefore, the polynomial in the numerator has smaller degree than the one in the denominator and the ratio is approaching~$0$ as~$n$ goes to infinity.
\end{proof}

\section{Conclusion and further directions}
\label{sec:conclusion}

In Section~\ref{sec:fixed-major}, we proved the monotonicity of the numbers $M^m_n(\sigma)$ for a single pattern~$\sigma$ other than $12 \cdots k$ (recall Theorem~\ref{thm:monotonity}) and showed an example of a set~$\Pi$ for which the monotonicity does not hold even though $M^m_n(\Pi)$ tends to infinity.
The natural question to ask would be whether we can in general characterize such sets~$\Pi$ for which the monotonicity of columns does not hold even though $deg(m, \Pi) \geq 1$.
Based on computing the values $M^m_n(\Pi)$ for small $n$ and various sets~$\Pi$, it seems to us that these cases are rather rare.

%We showed a counterexample to the idea that the columns are either weakly increasing or eventually zero. 
%On the other hand we think that this holds for permutations with a fixed core.
%
%\begin{conjecture}
%For a set of permutations~$\Pi$ and a core~$\gamma$ there is no~$n$ such that $M^{[\gamma]}_n(\Pi) > M^{[\gamma]}_{n+1}(\Pi)$ and $M^{[\gamma]}_{n+1}(\Pi) < M^{[\gamma]}_{n+2}(\Pi)$. 
%\end{conjecture}

In Section~\ref{sec:asymptotics}, we analysed the asymptotic behaviour of the numbers $M^m_n(\Pi)$ for many types of~$\Pi$ in the sense of the degree $deg(m, \Pi)$.
The most natural way to extend this study is to cover the remaining cases.
For example, it remains to be shown whether the sets~$\Pi$ that contain permutations with both finite and infinite magnitude obey any general rules.
Another open problem is to determine exactly for which sets~$\Pi$ the values $M^m_n(\Pi)$ are eventually equal to zero.

One could also focus on generalized pattern avoidance.
A permutation~$\sigma$ contains a copy of a generalized pattern~$\pi$ if it contains $\pi$ and certain elements of the diagram of the copy are adjacent either horizontally or vertically. 
The concept of generalized patterns was introduced by Babson and Steingrímsson~\cite{Babson2000}.
The reason they are interesting is because many statistics on permutations (including the number of inversions and the major index) can be expressed as a linear combination of the number of occurrences of these generalized patterns. 

Finally, similar analysis of the distribution could be done for other permutation statistics like number of descents or number of excedances.
As previously mentioned, the number of inversions was already covered by Claesson, Jelínek and Steingrímsson~\cite{Jelinek2012}. 
One can find examples of various other pattern statistics in a classification given by Babson and Steingrímsson~\cite{Babson2000}.
% is it long enough?

\bibliographystyle{plain}
\bibliography{mybib}

\end{document}